\documentclass[oneside,12pt]{article}
\usepackage[english]{babel}
\RequirePackage{latexsym} \RequirePackage{amsthm}
\RequirePackage{amsmath}

\RequirePackage{amssymb} \RequirePackage{makeidx}

\usepackage{enumerate}

\theoremstyle{definition}
\newtheorem{defin}{Definition}[section]
\newtheorem{prop}[defin]{Proposition}
\newtheorem{theorem}[defin]{Theorem}

\newtheorem{lemma}[defin]{Lemma}
\newtheorem{cor}[defin]{Corollary}
\newtheorem*{thm}{Main Theorem}
\newtheorem*{remark}{Remark}
\newtheorem*{Thm}{Theorem}
\theoremstyle{plain} \newtheorem{exa}[defin]{Example}

\RequirePackage{fancyhdr} 

\fancypagestyle{plain}{%
    \fancyhead{} %
    }

\usepackage{color}

\makeatletter
\def\cleardoublepage{\clearpage\if@twoside \ifodd\c@page\else
  \hbox{}
  \vspace*{\fill}
  \begin{center}
  \end{center}
  \vspace{\fill}
  \thispagestyle{empty}
  \newpage
  \if@twocolumn\hbox{}\newpage\fi\fi\fi}
\makeatother

\usepackage{pgf,tikz}
\usetikzlibrary{arrows}

\makeatletter
\def\cleardoublepage{\clearpage\if@twoside \ifodd\c@page\else
  \hbox{}
  \vspace*{\fill}
  \begin{center}
  \end{center}
  \vspace{\fill}
  \thispagestyle{empty}
  \newpage
  \if@twocolumn\hbox{}\newpage\fi\fi\fi}
\makeatother

\newcommand{\rank}{\mathop{\operator@font rank}}

\makeatother

\numberwithin{equation}{section}

\makeatletter
\renewcommand\section{\@startsection {section}{1}{\z@}%
                                   {-3.5ex \@plus -1ex \@minus -.2ex}%
                                   {2.3ex \@plus.2ex}%
                                   {\normalfont\large\bfseries}}
\renewcommand\subsection{\@startsection{subsection}{2}{\z@}%
                                     {-3.25ex\@plus -1ex \@minus -.2ex}%
                                     {1.5ex \@plus .2ex}%
                                     {\normalfont\bfseries}}
\makeatother

\makeindex

\begin{document}
\title{{Stabiliser of an Attractive Fixed Point of an IWIP Automorphism of a free product}}	

\author{Dionysios Syrigos}
\date{\today}
\maketitle

\begin{abstract}
	For a group $G$ of finite Kurosh rank and for some arbiratily free product decomposition of $G$, $G = H_1 \ast H_2 \ast ... \ast H_r \ast F_q$, where $F_q$ is a finitely generated free group, we can associate some (relative) outer space $\mathcal{O}(G, \{H_1,..., H_r \})$. We define the relative boundary $\partial (G, \{ H_1, ..., H_r \}) = \partial(G, \mathcal{O}) $ corresponding to the free product decomposition, as the set of infinite reduced words (with respect to free product length). By denoting $Out(G, \{ H_1, ..., H_r \})$ the subgroup of $Out(G)$ which is consisted of the outer automorphisms which preserve the set of conjugacy classes of $H_i$'s, we prove that for the stabiliser $Stab(X)$ of an attractive fixed point in $X \in \partial (G, \{ H_1, ..., H_r \})$ of an irreducible with irreducible powers automorphism relative to $\mathcal{O}$, it holds that it has a (normal) subgroup $B$ isomorphic to subgroup of $\bigoplus \limits_{i=1} ^{r} Out(H_i)$ such that $Stab(X) / B$ is isomorphic to $\mathbb{Z}$. The proof relies heavily on the machinery of the attractive lamination of an IWIP automorphism relative to $\mathcal{O}$.
\end{abstract}

	\newpage
	\tableofcontents
	\newpage
	
\section{Introduction}
The outer automorphism group $Out(F_n)$ of a finitely generated free group $F_n$ has been extensively studied. In particular, $Out(F_n)$ has been studied via its action on the outer space $CV_n$, which has been introduced by Culler and Vogtmann in \cite{culler1986moduli}. Let $G$ be a group of finite Kurosh rank, i.e. $G$ can be written as a finite free product of the form $G = G_1 \ast G_2 \ast ... \ast G_m \ast F_n$, where all the $G_i$'s are freely indecomposable and $F_n$ is a finitely generated free group. The concept of the outer space can be generalised for such a group $G$ and there is a contractible space of $G$-trees, $\mathcal{O} (G, \{G_1, G_2,...,G_m\}, F_n)$ on which $Out(G)$ acts. This space was introduced by Guirardel and Levitt in \cite{guirardel2007outer}. In fact, for a group $G$ as above and any non-trivial free product decomposition $G = H_1 \ast H_2 \ast ... \ast H_r \ast F_q$ (here $H_i$ may be freely decomposable or isomorphic to $\mathbb{Z}$), they constructed a relative outer $\mathcal{O} (G, \{H_1, H_2,...,H_r\}, F_q)$ on which the subgroup $Out(G, \{ H_1, ..., H_r \} )$ of $Out(G)$ acts, where $Out(G, \{ H_1, ..., H_r \} ) = \{ \Phi \in Out(G) |$ for every $i = 1,.., r$, there is some $j$ s.t. $ \phi(H_i) = g_i H_j g^{-1}  \}$.\\
For a finitely generated free group $F_n$, it is well known that we can define the boundary  $\partial F_n$ which is a Cantor set and it can be viewed as the set of infinite reduced words (for some fixed basis of $F_n$). Moreover, an automorphism $\phi \in Aut(F_n)$, can be seen as a quasi-isometry of $F_n$ and therefore it induces a homeomorphism of the boundary $\partial F_n$, which we denote by $\partial \phi$. As a consequence, we can study the subgroup $Stab(X)$ of automorphisms that fix the infinite word $X \in \partial F_n$, i.e. $
\phi \in Stab(X)$ iff $\partial \phi(X) = X$. In this paper, we would like to study the corresponding notions for a group $G$ (of finite Kurosh rank) relative to some fixed non-trivial free product decomposition and especially the  $Aut(G, \{H_1, H_2,...,H_r\} )$- stabiliser of infinite reduced words. Firstly, we fix the outer space corresponding to some free product decomposition of $G$, as above. In this case, we can define a (relative) boundary $\partial (G, \{ H_1, ..., H_r \})$ as the set of infinite reduced words for the free product length (for some fixed basis of the free group). Similarly, every $\phi \in Aut(G, \{ H_1, ..., H_r \} )$ induces a homeomorphism $\partial \phi$ of the relative boundary $\partial (G, \{ H_1, ..., H_r \})$.
It is natural to ask if we can compute the subgroups $Stab(X)$ for any $X \in \partial (G, \{ H_1, ..., H_r \}) $. However, there is no full answer even in the free case. There are some partial results and for example it is not difficult to see that the stabiliser of any point of the form $u^{\infty} = u u u ...$ (where $u$ is a hyperbolic cyclically reduced element of $G$),  is the subgroup of $Aut(G, \{ H_1, ..., H_r \})$, which is consisted of the automorphisms that fix $u$. There is also a result of Hilion in this direction (see \cite{hilion2007maximal}). An automorphism of $Aut(F_n)$ is said IWIP (i.e. irreducible, with irreducible powers), if no non-trivial free factor of $F_n$ is mapped by some power $k > 1$ to a conjugate of itself. Therefore, using this terminology, Hilion's result can be stated as:

\begin{Thm}\label{HILION}[Hilion] \cite{hilion2007maximal}
If $X \in \partial F_n$ is an attractive fixed point of an IWIP automorphism, then $Stab(X)$ is infinite cyclic.
\end{Thm}

In the general case, we say that an automorphism $\phi$ is IWIP relative to $\mathcal{O}$, if there is non-trivial free factor $B$ of $G$ that strictly contain some conjugate of some $H_i, i = 1,...,r$, it is mapped by some power $\phi ^{k}$, $k > 1$ to a conjugate of itself.
Then the main result of the present paper is a generalisation of the previous theorem.\\
 However, here there is a difference that arises from the factor automorphisms of the $H_i$'s. More precisely:

\begin{thm} \label{MAIN}
	If $X \in \partial  (G, \{ H_1, ..., H_r \})$ is an attractive fixed point of an IWIP automorphism $\phi$, then $Stab(X)$ has a subgroup $B$ isomorphic to a subgroup of  $\bigoplus \limits_{i=1} ^{p} Out(H_i)$ and $ Stab(X) / B$ is infinite cyclic.
\end{thm}

In our case, there are examples of $X$, as above, where $Stab(X)$ is not infinite cyclic. We describe such an example in the last section, see \ref{example}. The main idea is that there are attractive fixed words of IWIP automorphisms that contain even finitely many elements of the elliptic free factors. Moreover, if a factor automorphism (an automorphism of some $H_i$) fixes these words, then it stabilises the attractive fixed point. So if $H_i$'s are sufficiently big, there are non-trivial automorphisms of $H_i$ that fix these words. As a consequence, we can find arbitrarily large subgroups of $Stab(X)$. On the other hand, if we suppose that every $Out(H_i)$ is finite, then we have a similar result as in the free case. In particular:
\begin{cor}
	If $X \in \partial  (G, \{ H_1, ..., H_r \})$ is an attractive fixed point of an IWIP automorphism $\phi$ and every $Out(H_i)$ is finite, then $Stab(X)$ is virtually cyclic.
\end{cor}
Our proof is similar to that of \cite{hilion2007maximal}, but we have to adjust the notions and to use the generalisations of the results used by Hilion, for the general case of free products. In particular, we use the work of Francaviglia and Martino \cite{francaviglia2013stretching} for train track representatives of IWIP automorphisms of a free product instead of the classical notion of train track representatives of automorphisms of free groups \cite{bestvina1992train}.
In fact, here an IWIP automorphism relative to $\mathcal{O}$ can be represented by a train-track map which is a $G$-equivariant, Lipschitz map $f : T 
\rightarrow T$, where $T \in \mathcal{O}$ and $f(gx) = \phi(g) f(x)$, with the property that no backtracking subpath occurs if one
iterates the train-track map on any edge of $T$. As a consequence of the general notion of train-track representatives, the author in \cite{syrigos2014irreducible} generalised the work of Bestvina, Handel and Mosher in the free case \cite{bestvina1997laminations}, and in particular we have the notion of the attractive lamination of an IWIP automorphism.\\
Now let us describe the basic steps of the proof.
Fistly, we construct a nice splitting of a given attractive fixed point $X$ of an IWIP automorphism, using train track representatives, which
matches the language of the attractive lamination.
Then we relate the subgroup $Stab(X)$ to the stabiliser of the attractive lamination, and so using the fact proved in \cite{syrigos2014irreducible} about the stabiliser of the attractive lamination and a technical lemma, and more specifically the fact that $Stab(X) \cap Out(G, \{H_i \} ^{t}) $ is torsion free, we get the main result.\\
As we have seen, there are a lot of facts that they are shard by $CV_n$ and the general space $\mathcal{O}$. As we have already mentioned that the train track representatives can be generalised in the general case. In the same paper, there is the construction and the properties of the Lipschitz metric for $\mathcal{O}$ which is a metric that the same authors previously studied for $CV_n$ (see \cite{francaviglia2011metric}). Recently, there are more papers that they indicate that we can find more similarities between $CV_n$ and $\mathcal{O}$. For example, the construction of hyperbolic spaces on which $Out(G, \{ H_1, ..., H_r \})$ acts (\cite{handel2014relative}), \cite{horbez2014hyperbolic}), the boundary of outer space (\cite{horbez2014boundary}), the Tits alternative for subgroups of $Out(G)$ (\cite{horbez2014tits}), the study of the asymmetry of the Lipschitz metric(\cite{syrigos2015asymmetry}) and the study of the centralisers of IWIP automorphisms(\cite{syrigos2014irreducible}).\\
OUTLINE: In Section 2, we recall some useful definitions and facts, in additional we generalise some well known notions for free groups to the free product case and we prove some basic preliminary results that we need for the main theorem.  In Section 3, we describe the construction of the attractive lamination for an IWIP automorphism and we list some properties. The last section is devoted to the proof of the main theorem.

\section{Preliminaries}
\subsection{$\mathbb{R}$ -trees, Kurosh rank}

Let $G$ be a group of \textit{ finite Kurosh rank} i.e $G$ splits as a free product $G = H_{1} \ast ...\ast H_{s} \ast F_{r}$, where every $H_i$ is non-trivial, not isomorphic to $\mathbb{Z}$ and freely indecomposable. Here the Kurosh rank of $G$ is just the number $s+r$. This decomposition is called the \textit{Gruskho decomposition}. It is the "minimal" decomposition of $G$ and it is unique, in the sense that the free rank $r$ is well defined and the $H_i$'s are unique up to conjugation. The class of groups of finite Kurosh rank contain strictly the class of finitely generated groups (by the Grushko theorem). We are interested only for groups which have finite Kurosh rank.
\\
In particular, for such a group $G$ we fix an \textit{arbitary} (non-trivial) free product decomposition $G = H_{1} \ast ...\ast H_{m} \ast F_{n}$ (i.e. we don't assume that every $H_i$ is not infinite cyclic or even freely indecomposable). However, we usually assume that $m+n > 2$. These groups admit co-compact actions on $\mathbb{R}$-trees (and vice-versa). It is useful that we can also apply the theory in the case that $G$ is free, and the $H_i$'s are certain free factors of $G$ (relative free case).\\

We consider isometric actions of the group $G$ on $\mathbb{R}$-trees induced by the free product decomposition and, more specifically, we say that $T$ is a \textit{$G$-tree,} if it is a simplicial metric tree $ (T, d_T )$, where $G$ acts simplicially on $T$ (sending vertices to vertices and edges to edges) and for all $g \in G, e \in E(T) $ we have that $e$ and $ge$ are isometric. Moreover, we suppose that every $G$-action is \textit{minimal}, which means that there is no $G$-invariant proper subtree.\\
Now let's fix a $G$-tree $T$. An element $g \in G$ is called \textit{hyperbolic}, if it doesn't fix any points of $T$. Any hyperbolic element $g$ of $G$ acts by translations on a subtree of $T$ homeomorphic to the real line, which is called the axis of $g$ and is denoted by $axis_T (g)$. The \textit{translation
	length} of $g$ is the distance that $g$ translates its axis. The action of $G$ on $T$
	defines a length function denoted by
	\begin{equation*}
	\ell_T : G \rightarrow R, \ell_T (g) : = \inf \limits_{x \in T} d_T (x, gx).
	\end{equation*}
	In this context, the infimum is always minimum and we say that $g \in G$ is hyperbolic if and only if $\ell_T (g) > 0$. Otherwise, $g$ is called \textit{elliptic} and it fixes a (unique) point of $T$.
	For more details about group actions on $\mathbb{R}$-trees, see \cite{culler1987group}.

\subsection{Relative Outer Space}
In this subsection we recall some basic definitions and properties. More details about the relative outer space can be found in \cite{francaviglia2013stretching}.\\
We consider $G$-trees as in the previous subsection.
We will define an outer space $\mathcal{O} = \mathcal{O}(G, (H_i)^{m}_{i=1}, F_n)$ relative to some fixed free product decomposition of $G$. More specifically, the elements of the outer space can be thought as simplicial metric $G$-trees, up to $G$-equivariant homothety.
Moreover, we require that these $G$-trees also satisfy the following conditions:
\begin{itemize}
	\item The action of $G$ on $T$ is minimal.
	\item The edge stabilisers are trivial.
	\item There are finitely many orbits of vertices with non-trivial stabiliser, more precisely for every $H_i$, $i = 1,..., m$ (as above) there is exactly one vertex $v_i$ with stabiliser $H_i$ (all the vertices in the orbits of $v_i$'s are called \textit{non-free vertices}).
	\item All other vertices have trivial stabiliser (and we call them \textit{free vertices}).
	\item The quotient $G / T$ is a finite graph of groups
\end{itemize}
Note that the last condition follows from the others, but we mention it in order to emphasise the importance of the co- compactness of the action.

\underline{\textbf{Action}}: Let $Aut(G, \mathcal{O})$ be the subgroup of $Aut(G)$  that preserve the set of conjugacy classes of the $H_i$ 's. Equivalently, $\phi \in Aut(G)$ belongs to $Aut(G,\mathcal{O})$ iff $\phi(H_i)$ is conjugate to one of the $H_j$ 's (in general, $i$ may be different to $j$). The group $Aut(G,\mathcal{O})$ admits a natural action on a simplicial tree by "changing the action", i.e. for $\phi \in Aut(G, \mathcal{O})$ and $T \in \mathcal{O}$, we define $\phi(T)$ to be the metric tree with $T$, but the action is given by $g*x = \phi(g)x$ (where the action in the right hand side is the action of the $G$-tree $T$). As $Inn(G)$ acts on $\mathcal{O}$ trivially, $Out (G,\mathcal{O}) = Aut(G,\mathcal{O})/ Inn(G)$ acts on $\mathcal{O}$. Note also that in the case of the Grushko decomposition, we have $Out(G) = Out(G,\mathcal{O})$.\\


\begin{remark}
Note that for a $g \in G$ and $T, S \in \mathcal{O}$, it holds that $g$ is hyperbolic relative to $T$ iff $g$ is hyperbolic relative to $S$. Therefore it makes sense to say that $g$ is hyperbolic relative to $\mathcal{O}$, as we will do. We denote by $Hyp(\mathcal{O})$ the set of hyperbolic elements of $\mathcal{O}$.
\end{remark}

\subsection{Topological Representatives, $\mathcal{O}$ - Maps}

\underline{\textbf{Edge paths:}}
Firstly, we would like to define the notion of an edge path for some tree $T \in \mathcal{O}$. More specifically, since $T$ is an $\mathbb{R}$-tree we have that any edge is isometric to the interval $[0,\ell(e)]$.
We say that an edge path is a reduced path of the form $e_1e_2...e_n$ (without backtracking). We can also define an infinite edge path, as an infinite reduced path of the form $e_1 e_2...e_n e_{n+1}...$. Similarly, we can define a bi-infinite edge path. We usually call paths lines these paths.\\
\underline{\textbf{Tightening:}} Every path $p$ is homotopic (relative endpoints) to a unique edge path $[p]$ in $T$. Actually, we can obtain from $p$ the path $[p]$, after removing the backtracking, and we say that $[p]$ is obtained by tightening $p$.\\
\\
\underline{\textbf{$\mathcal{O}$ - maps:}}
\begin{defin}
We say that a map between trees $A,B \in \mathcal{O}$, $f : A \rightarrow B$ is an \textbf{$\mathcal{O}$- map}, if it is a $G$-equivariant, Lipschitz continuous, surjective function.
\end{defin}

It is very useful to know that there are such maps between any two trees. This is true and, additionally, by their construction they coincide on the non - free vertices. More specifically, by \cite{francaviglia2013stretching} we get:
\begin{lemma}\label{O-maps}
For every pair $A,B \in \mathcal{O}$; there exists a $\mathcal{O}$-map $f : A \rightarrow B$.
Moreover, any two $\mathcal{O}$-maps from $A$ to $B$ coincide on the non-free vertices.
\end{lemma}

Now we will prove that every $\mathcal{O}$-map is a quasi-isometry. But firstly, we need a technical lemma:

\begin{lemma}\label{inclusion}
Let $T \in \mathcal{O}$ and $v$ be a vertex of $T$. Then the inclusion map $\iota$ from the $G$-orbit of $v$, $A= G \cdot b$ to $T$ is a quasi-isometry. As a consequence, any projection $p$ from $T$ to $A$ is again a quasi-isometry.
\end{lemma}
\begin{proof}
It is obvious that the inclusion map is $1-1$ and satisfies that $d_A (x, y) = d_T(x,y)$. So it remains to show that $\iota$ is quasi- onto, which means that there is some $M$ s.t. for every $x \in T$ there is some $g \in G$ with $d_T(x, g v) \leq M$. This follows from the fact that the quotient $\Gamma = G / T$ is compact and therefore we can choose $M$ to be the maximum distance in $\Gamma$ between the projection of $v$ and the other vertices. Therefore the result follows.
\end{proof}

\begin{lemma}\label{quasi-isometry}
Let $T, S \in \mathcal{O}$ and $f : T \rightarrow S$ be an $\mathcal{O}$-map. Then $f$ is a quasi-isometry.
\end{lemma}
\begin{proof}
Let choose some vertex $v \in T$, then $f$ induces a Lipschitz map from $A = G v$ to $B = G f(v) $. Note also that by the construction of $\mathcal{O}$-maps there is an $\mathcal{O}$ - map $h$ from $S$ to $T$ which is the inverse function of $f_{|A}$ restricted to $B$ and it is again Lipschitz. Therefore $f_{|A}$ is an isomorphism between $A$ and $B$ (and in particular quasi-isometry).\\
Using now the lemma \ref{inclusion} and the fact that the inverse of a quasi-isometry is a quasi-isometry we get that:
$T \xrightarrow {p} A \xrightarrow{f} B \xrightarrow{q} S $, where $p$ is the projection of $T$ to $A$ and $\iota$ is the  inclusion map, where the maps $p,q,f_{|A}$ are quasi-isometries, and it follows that $f$ is a quasi- isometry.
\end{proof}

Using the lemma \ref{quasi-isometry} and the existence of $\mathcal{O}$-maps between every two elements of $\mathcal{O}$ (see \ref{O-maps}), we get that:
\begin{prop}\label{trees q.i.}
	Let $T, S \in \mathcal{O}$, then we have that the metric trees $T$ and $S$ are quasi -isometric.
\end{prop}

\underline{\textbf{Topological representatives:}} It is very useful to see an outer automorphism as a map between a tree $T \in \mathcal{O}$. More specifically:
\begin{defin}
Let $\Phi \in Out(G, \mathcal{O})$ and $T \in \mathcal{O}$, then we say that a Lipschitz surjective map $f : T \rightarrow T$ \textbf{represents}
$\Phi$ if for any $g \in G$ and $t \in T$ we have $f(gt) = \Phi(g)(f(t))$. In other words, $f$ is an $\mathcal{O}$-map from $T$ to $\Phi(T)$.
\end{defin}

Applying again \ref{O-maps} (the existence of $\mathcal{O}$-maps), we get:

\begin{lemma}
Let $\Phi \in Out(G, \mathcal{O})$ and $T \in \mathcal{O}$. Then there is a (simplicial) topological representative of $\Phi$ in $T$.
\end{lemma}

The topological representatives many times produce paths which are not reduced, and then we have cancellation in their images. Therefore we have to define a map (induced by $f$) from the reduced edge-paths of $T$ to itself,  and we denote it by $f_{\#}$, by the rule $f_{\#} (w) = [f(w)]$ for every edge-path $w$ of $T$.
\\

As these maps represent an outer automorphism $\Phi$, if we change the tree $T$ with $\iota_h (T)$ where $\iota_h \in Inn(G) $ is just the conjugation by some $h \in G$, we get an other $\mathcal{O}$-map that still represents $\Phi$. Therefore each regular automorphism $\phi \in \Phi$ corresponds to some topological representative. In particular, for a topological representative $f : T \rightarrow T$ of an automorphism $\Phi$, and for every $\phi \in \Phi$, changing appropriately the tree $T$ with some $\iota_h (T)$, we can choose $f$ to satisfy $\phi(g) f = f g$, for every $g \in G$. We say that $f, \phi$ are \textbf{mated}. Note that in this case, for $g \in G$ (as we can see $g$ as an isometry of $T$):\\
\begin{remark}
$g \in Fix(\phi)$ if and only $g$ and $f$ commute.
\end{remark} 
\subsection{N-periodic Paths}
Here we will define the notion of an $N$ - path. See more about the properties of $N$-periodic paths in \cite{syrigos2014irreducible}.
A difference between the free and our case is that it is not always true that there are finitely many orbits of paths of a specific length (if there are non-free vertices with infinite stabiliser), but it is true that there are finitely many paths that have different projections in the quotient $G / T$. Therefore the notion of an N- path (we define it below) plays the role of a Nielsen path. Note that here if $h: S \rightarrow S$, we say that a point $x \in S$ is $h$-periodic, if there are $g \in G$ and some natural $k$  s.t. $h^{k}(x) = gx$.

\begin{defin}
	\begin{enumerate}
		\item Two paths $p, q$ in $S \in \mathcal{O}$ are called \textit{equivalent}, if they project to the same path in the quotient $G/S$. In particular, their endpoints $o(p), o(q)$ and $t(p), t(q)$ are in the same orbits, respectively.
		\item Let $h : S \rightarrow S$ be a representative of some outer automorphism $\Psi$, let $p$ be a path in $S$ and let's suppose that the endpoints of $p$ and $h(p) $ are in the same orbits (respectively), then we say that a path $p$ in $S$ is \textit{N-path} (relative to $h$), if the paths $[h(p)], p$ are equivalent.
	\end{enumerate}
\end{defin}

In the free case, we need representatives of outer automorphisms for which we can control the number of Nielsen paths. For example the notion of stable and appropriate train track representatives. As we will see, we can define the corresponding notions in our case but using N-paths.

\subsection{Relative Boundary}
Let's fix some relative outer space $\mathcal{O}$ with respect to a some fixed free product decomposition of $G$. We will give the definition of the relative boundary relative to $\mathcal{O}$.\\
For every $T \in \mathcal{O}$, we can use the Gromov hyperbolic boundary $\partial T$, as $T$ is a $0$-hyperbolic space (or a tree), by defining it as the set of equivalence classes of sequences of points in $T$ that converge to infinity with respect to the Gromov product (with respect to some fixed base point $p$).
However, it is more convenient for our purposes to define it as the set of lines passing through a base point $x \in T$. The two definitions coincide in the case of a proper (i.e. the closed balls are compact) hyperbolic metric space. But in the case of trees, we don't need the properness. For more details about the Gromov Boundary, see the very interesting survey for boundaries of hyperbolic spaces \cite{kapovich2002boundaries}. \\
More specifically, for any two lines $\ell, \ell'$ starting from $x \in T$, we define the equivalence relation by $\ell \equiv \ell ' $ iff $\ell, \ell '$ have an infinite common subline. Now we denote the boundary by $\partial _x T = \{ [\ell] | \ell : [0,\infty) \rightarrow T$ is a geodesic ray with $\ell (0) = x \}$.  It is not difficult to see that this definition does not depend on the base point and so we will usually omit the base point from the notation.\\
We can also define the $r$ neighbourhood of a point $r$ in the boundary, as $V (p, r) = \{ q \in \partial _x	T |$ for any geodesic rays $\ell _1, \ell_2$ starting at $x$ and with $[\ell_1] = p , [\ell_2] = q$ we have $\liminf \limits_{n \rightarrow \infty}	|\ell _1(t) \wedge \ell_2(t)| \geq r \}$.\\
Let $p, q \in V(T)  \cup  \partial T$, we define the operation $\wedge$ as follows: $p \wedge q$ is the common initial subpath (starting from $x$) of the unique edge paths $[x,p], [x,q]$ that connect $p,q$ with the base point $x$.

We now topologise $\partial  T$ by setting the basis of neighborhoods for any $p \in \partial
T$ to be the collection $\{ V (p, r) | r \geq 0 \}$.
Moreover, this topology is metrisable and in particular, the metric on $\partial T$ is given by $d(p, q) = e^{ - |[x,p] \wedge [x,q]|}$ for $p, q \in \partial T$ (where $e^{-  \infty} = 0$).\\
It is not difficult to see that any quasi-isometry $f: T \rightarrow S$, induces a homeomorphism  between the boundaries $\partial T, \partial S$, as constructed. In particular, since any $\mathcal{O}$-map $f: T \rightarrow S$ is a quasi-isometry, it can be extended to the boundary and it induces a well defined homeomorphism, which we denote by $\partial f : 
\partial T \rightarrow \partial S$. Therefore we get that:
\begin{lemma}
	Let $T,S \in \mathcal{O}$. Then $\partial T$ is homeomorpic to $\partial S$.
\end{lemma}
Note that in our case, if there is some infinite $H_i$ it is easy to see that $\partial T$ is not compact in the metric topology. For example, if we have a point of infinite valence we can produce a sequence of lines that they have constant distance between each other. Therefore we have a sequence in $\partial T$, which has not converging subsequence.
However, it is possible to find other interesting topologies for which $T \cup \partial T$ is compact. For example, see \cite{coulbois2007non} for the observers' topology.

We can also define the set $\partial  (G, \mathcal{O})$ of infinite reduced words with respect to the free product length which is induced by our fixed free product decomposition.  For any $A, B \in G  \cup  \partial (G, \mathcal{O})$, we define the operation $\wedge$ as follows: $A \wedge B$ is the longest common initial subword of $A,B$. It is easy to see that the map $d(A, B) = e^{-|A \wedge B|}$, for $A \neq B$ and $d(A,A) =0$ is a metric on the space $G \cup \partial G$. Finally, since any $\phi \in Aut(G, \mathcal{O})$ can be seen as a quasi-isometry of $G$, we have that it induces a homeomorphism of $\partial (G, \mathcal{O})$ which we denote by $\partial \phi$.
Note that the two notions of the boundary can be identified, in particular:
\begin{lemma}
Let $T \in \mathcal{O}$. Then $\partial T$ is homeomorphic to $\partial (G, \mathcal{O})$.
\end{lemma}

\begin{proof}
Consider the universal cover $S$ of the rose of $m$ cycles with $n$ edges attached, corresponding to the free product decomposition $G =  H_{1} \ast ...\ast H_{m} \ast F_{n}$ with length of edges corresponding to the $n$ simple loops to be $1$ and of the rest of edges to have length $1/2$. It is easy to see now that of an edge path starting from a base point $v$ (it can be chosen to be the lift of the unique free vertex of the quotient) correspond to a word in $G$ (and vice versa) and the length of the edge path is exactly the free product length of the word. Moreover, the lines starting from the base point correspond to infinite reduced words of $G$ with respect to the free product length. Therefore there is a bijection from the set $\partial S$ of lines of $S$ starting from $v$ to $\partial (G, \infty)$. Since the metrics are the similar, it is easy to see that this map is a actually a homeomorphism.
 But now for every $T \in \mathcal{O}$, we have that $\partial T$ is homeomorphic to $\partial S$ and the lemma follows.
\end{proof}

Note also that since $\partial T, \partial (G, \mathcal{O})$ are homeomorphic, we can identify $\partial f$ and $\partial \phi$.

\subsection{Rational and non-Rational Points}

For every hyperbolic element $g$ of $\mathcal{O}$, the sequence of elements $g^{k}$ have arbitrarily large (free product) length and so it has a limit in the relative boundary  $\partial (G, \mathcal{O})$ , which we denote it by $g^{\infty}$. We can also define $g^{- \infty}$ as $(g^{-1}) ^{\infty}$.
\\
If $g \in G$, we denote by $\iota _u$ the inner automorphism of $G$ given by $\iota _{u} (g) = u g u^{-1}$, for every $g \in G$. If $u \in Hyp(\mathcal{O})$ , it is easy to see that then $\partial \iota _u$ fixes exactly two points of the relative boundary $\partial (G, \mathcal{O})$ and more specifically the points $u ^{\infty}  , u ^{- \infty}$. Note that since edge stabilisers of elements of $\mathcal{O}$ are trivial, for an elliptic element $u$ then the inner automorphism $\iota _u$ cannot fix a point of the boundary.\\
We say that infinite words of the form  $u ^{\infty}  , u ^{- \infty}$ for a hyperbolic element $u$, are \textbf{rational points} of the boundary. Alternatively, we could define the rational points as the fixed points of inner automorphisms corresponding to hyperbolic elements.

\begin{prop}\label{injective}
If $X \in \partial  (G, \mathcal{O})$ is not a rational point, then the restriction of the quotient map $Aut(G, \mathcal{O}) \rightarrow Out(G, \mathcal{O})$ to $Stab(X)$ is injective.
\end{prop}

\begin{proof}
Let $T \in \mathcal{O}$ and let's assume that $\phi \in Aut(G, \mathcal{O})$.  Suppose  also that $X$ is a fixed point of $\partial  \phi$.\\
Let $u$ be a non-trivial of $\mathcal{O}$ and suppose that $X$ is fixed by $\partial (i_u \circ \phi)$, which implies that it is a fixed point of $\partial i_u$.  As a consequence, $u$ is a  hyperbolic element and in particular $X$ it is rational point or equivalently the axis of the hyperbolic element $u$, and so $X = u ^{\infty}$ or $X = u^{ - \infty}$, which leads us to a contradiction.
\end{proof}

\subsection{Regular and Singular Fixed Points}
For an automorphism $\phi \in Aut(G, \mathcal{O})$, we denote by $Fix \phi $ the fixed subgroup of $\phi$: $Fix \phi = \{ g \in G | \phi(g) = g \}$. Since by \cite{collins1994efficient}, 
we have that $Fix \phi$ has finite Kurosh rank, its (relative) boundary $\partial Fix \phi$ embeds into $\partial (G, \mathcal{O})$ and it is actually a subgroup of $Fix \partial \phi$ of fixed infinite words by $\partial \phi$. We will use the same terminology as in the free case and we distinguish two cases for infinite fixed words of an automorphism $\phi \in Aut(G,\mathcal{O})$ i.e. the elements of $Fix \partial \phi$ : either it belongs to $\partial Fix \phi$ and then it is called \textbf{singular}, or otherwise it is called \textbf{regular}.\\
For the singular fixed points, the are two subcases. We use the notion topologically attractive (and repulsive) fixed points as in \cite{martino1998automorphisms}, but there is also the a metric notion. These definitions are different in the free product case, while they coincide in the free case. For more details, see \cite{martino1998automorphisms}.\\
We say that a fixed point $X$ of $\partial \phi$ \textbf{attractive}, if there is an integer $N$ s.t. if $|Y \wedge X| \geq N$, then $\lim\limits_{n \rightarrow \infty} \phi ^{n} (Y) = X$. A fixed point $X$ is said to be \textbf{repulsive}, if it is attractive for $\partial \phi ^{-1}$.
A classification of fixed points of $\partial \phi$ has been proved in the proposition 5.1.14. of \cite{martino1998automorphisms} and more specifically:
\begin{prop}
Let $\phi \in Aut(G, \mathcal{O})$. A fixed point of $\partial \phi$ is :
\begin{itemize}
	\item either singular
	\item or attractive
	\item or repulsive
\end{itemize}
\end{prop}

\subsection{Bounded Cancellation Lemma}

Let $T, T′ \in \mathcal{O}$ and $f : T \rightarrow T'$ be an $\mathcal{O}$- map. If we have a concatenation of paths $ab$, ever if $f(a) = f_{\#} (a)$ and $f(b) = f_{\#} (b)$, it is possible to have cancellation in $f(a)f(b)$. However, the cancellation is bounded above by some some $M$ which depends only on $f$ and not on $a,b$ . In particular,  we can define the bounded cancellation constant of $f$ (let's denote it $BCC(f)$) to be the supremum of all real numbers $N$ with the property that there exist $A, B, C$ some points of $T$ with $B$ in the (unique) reduced path between $A$ and $C$ such that $d_{T′}(f(B), [f(A), f(C)]) = N$ (the distance of $f(B)$ from the reduced path connecting $f(A)$ and $f(C)$ ), or equivalently is the lowest upper bound of the cancellation for a fixed $\mathcal{O}$-map.\\
The existence of such number is well known, for example a bound has given in \cite{horbez2014hyperbolic}:
\begin{lemma} \label{BCL}
	Let $T \in \mathcal{O}$, let $T′ \in \mathcal{O}$, and let $f : T \rightarrow T'$ be a Lipschitz map. Then $BCC(f) \leq Lip(f) qvol(T)$, where $qvol(T)$ the quotient volume of $T$, defined as the infimal volume of a finite subtree of $T$ whose $G$-translates cover $T$.
\end{lemma}
Therefore we can define a new map, in particular:
\begin{defin}
Let $f: T \rightarrow T$ be a topological representative of $\Phi \in Out(G, \mathcal{O})$ and let's denote by $C$ the bounded cancellation lemma of $f$. Then for every edge path $w$ of $Τ$, we can define the map $f_{\#, C} (w)$ as the path obtained by removing both extremities of length $C$ from the reduced image of $f_{\#}$.
\end{defin}

\subsection{Train Tracks}
In this section we will define the notion of a "good" representative of an outer automorphism $\Phi \in Out(G, \mathcal{O})$. As we have seen there are representatives of every outer automorphism (i.e. $\mathcal{O}$-maps from $T$ to $\phi(T)$), but sometimes we can find representatives with better properties. In particular, we want a topological representative $f$, where $f^{k} (e) =  f^{k} _{\#} (e) $ for every $k$ and for every edge $e$. These maps, which are called \textit{train track maps}, are very useful and every irreducible automorphism has such a representative (we can choose it to be simplicial, as well).\\
We give below a more general definition of a train track map representing an outer automorphism. We are interested for these maps because we can control the cancellation (as we have seen, it is not possible to avoid it).
Firstly, we need the notions of a legal path relative to some fixed train track structure.
\begin{defin}
	\begin{enumerate}
		\item A \textbf{pre-train track structure} on a $G$-tree $T$ is a $G$-invariant equivalence relation on the set of germs of edges at each vertex of $T$.
		Equivalence classes of germs are called \textbf{gates}.
		\item A \textbf{train track structure} on a $G$-tree T is a pre-train track structure with at least two gates at every vertex.
		\item A \textbf{turn} is a pair of germs of edges emanating from the same vertex. A \textbf{legal turn} is called a turn for which the two germs belong to different equivalent classes. A \textbf{legal path}, is a path that contains only legal turns.
	\end{enumerate}	
\end{defin}

Now we can define the train track maps.

\begin{defin}
	An $\mathcal{O}$-map $f : T \rightarrow T$, which is representing $\Phi$ is called a train track map, if there is a train track structure on $T$ so that
	\begin{enumerate}
		\item  f maps edges to legal paths (in particular, $f$ does not collapse edges).
		\item If $f(v)$ is a vertex, then $f$ maps inequivalent germs at $v$ to inequivalent germs at $f(v)$.
	\end{enumerate}
\end{defin}

However, we can not have such representatives for any outer automorphism. But it can be proved that for an interesting class of outer automorphisms can be represented by such a map. We will describe this class for regular automorphisms, but it can easily be defined for outer automorphisms as well.\\
In the free case, an  automorphism $\phi \in Aut(G, \mathcal{O})$ is called  \textit{irreducible}, if it there is no $\phi$-invariant free factor up to conjugation. In our case we know that the $H_i$'s are invariant free factors, but we don't want to have "more invariant free factors". More precisely,
we will define the irreducibility of some  automorphism \textit{relative} to the space $\mathcal{O}$ or to the free product decomposition. Similarly, we can define irreducibility for outer automorphisms of $G$.

Firstly, we will give the algebraic definition, but we need the notion of a free factor system. Suppose that $G$ can be written as a free product, $G = G_1 \ast G_2 \ast ...G_k \ast F_n$. Then we say that the set $\mathcal{A} = \{ [G_i] : 1 \leq i  \leq k \}$ is a\textbf{ free factor system for $G$}, where $[A]$ = $\{ gAg^{-1} : g \in G \}$ is the set of conjugates of $A$.\\
Now we define an order which we denote by $\sqsupseteq$ on the set of free factor systems of $G$. More specifically, given two free factor systems $\mathcal{G} = \{ [G_i] : 1 \leq i  \leq k \}$ and $\mathcal{H} = \{ [H_j] : 1 \leq j  \leq m \}$, we write $\mathcal{G} \sqsubseteq \mathcal{H}$ if for each $i$ there exists a $j$ such that $G_i \leq gH_jg^{-1}$ for some $g \in G$. The inclusion is strict, and we write $\mathcal{G} \sqsubset \mathcal{H}$, if some $G_i$ is contained strictly in some conjugate of $H_j$. We can see  $ \{[G] \} $ as a free factor system and in fact, it is the maximal (under $\sqsubseteq$) free factor system. Any free factor system that is contained strictly to $\mathcal{G}$ is called \textbf{proper}. Note also that the Grushko decomposition induces a free factor system, which is actually the minimal free factor system (relative to $\sqsubseteq$).\\
We say that $\mathcal{G} = \{ [G_i] : 1 \leq i  \leq k \}$ is $\phi$ - \textbf{invariant} for some $\phi \in Aut(G)$, if $\phi$ preserves the conjugacy classes of $G_i$'s. In each free factor system $\mathcal{G} = \{ [Gi] : 1 \leq i  \leq p \}$, we associate the outer space $\mathcal{O} = \mathcal{O}(G, (G_i)^{p}_{i=1}, F_k)$ and any $\phi \in Out(G)$ leaving $\mathcal{G}$ invariant, will act on $\mathcal{O}$ in the same way as we have described earlier.
\begin{defin}
	Let $\mathcal{G}$ be a free factor system of $G$ which is $\Phi$- invariant for some $\Phi \in Out(G)$. Then $\Phi$ is called \textit{irreducible relative to $\mathcal{G}$}, if $\mathcal{G}$ is a maximal (under $\sqsubseteq$) proper, $\Phi$-invariant free factor system.
\end{defin}

We could alternatively define the notion of irreducibility as:

\begin{defin}
	We say $\phi \in  Aut(G, \mathcal{O})$ is $\mathcal{O}$-\textit{irreducible}  if for any $T \in \mathcal{O}$ and choose some $f : T \rightarrow T$ representing $\Phi$, where $\phi \in \Phi$ and $f$ mated with $\phi$, if $W \subseteq T$ is a proper $f$-invariant $G$-subgraph then $W$ does not contain the axis of a hyperbolic element.
\end{defin}

The next lemma confirms that the two definitions of irreducibility are related.

\begin{lemma}
	Suppose $\mathcal{G}$ is a free factor system of $G$ with associated space of trees $\mathcal{O}$, and further suppose that $\mathcal{G}$ is $\phi$-invariant. Then $\phi$ is irreducible relative to $\mathcal{G}$ if and only if $\phi$ is $\mathcal{O}$-irreducible.
\end{lemma}

Now let's give the definition of an irreducible automorphism with irreducible powers relative to $\mathcal{O}$, which are the automorphisms that we will study.
\begin{defin}
	An outer automorphism $\phi \in Out(G, \mathcal{O})$ is called \textbf{IWIP} (\textit{irreducible with irreducible powers} or \textit{fully irreducible}), if every $\phi ^{k}$ is irreducible relative to $\mathcal{O}$.
\end{defin}

Now as we have said above, every irreducible outer automorphism has a train track representative. This fact it generalises the well known theorem of Bestvina and Handel (see \cite{bestvina1992train}) . In particular, we can apply it on every power of some IWIP automorphism.
\begin{theorem} [Francaviglia- Martino]
	Let $\Phi \in Out(G, \mathcal{O})$ be irreducible. Then there exists a simplicial train track map representing $\Phi$.
\end{theorem}

An interesting remark is the following:
\begin{remark}
	Every outer automorphism $\phi \in Aut(G)$ is irreducible relative to some appropriate space (or relative to some free product decomposition). Moreover, there are two cases: either $\phi$ is IWIP relative to $\mathcal{O}$ or it fixes a point of $\mathcal{O}$ (i.e. there is $T \in \mathcal{O}$ s.t. $\phi$ can be seen as an isometry of $T$).\end{remark}
In particular, using the remark above, in the relative free case we have some results for automorphisms of $Out(F_n)$ that they are not IWIP relative to $CV_n$, but they are IWIP relative to some appropriate space $\mathcal{O}$.

\underline{\textbf{Splittings and Appropriate train track maps:}}\\
In this section, we will define the notion of an appropriate train track representative which is similar to the definition of the free case. As we have discussed the notion of a Nielsen path in our case it has to be replaced by the notion of a $N$-path. Using $N$-paths, we can define the stable train track representatives.
\begin{defin}
We say that a train track representative $f$ of an outer automorphism $\Phi$ is stable, if it supports at most one equivalence class of $N$-paths.
\end{defin}
It is well known that every outer automorphism can be represented by a stable train track representative (for example see \cite{collins1994efficient} or \cite{sykiotis2004stable}). Let's denote by $p$ some representative of the unique class of the $N$-path that $f$ supports, if it exists. Here we need some even better notion of train track representatives, but firstly we need the notion of a splitting.\\
More specifically, let $f$ as above, and let $w$ be a path in $T$. We say that $w = ...w_m w_{m+1}...$, where $w_i$'s are non-trivial subpaths of $w$, is a splitting for $f$ if for all $k \geq 1$, $f^{k} _{\#} (w) = ...f^{k} _{\#} (w_m) f^{k} _{\#} (w_{m+1}).... $. Then we use the notation: $w =  ... \cdot w_m \cdot w_{m+1} \cdot ...$ and the $w_i$'s are called the bricks of $w$.
\begin{defin}
A stable train track representative  $f : T \rightarrow T$ of an IWIP outer automorphism $\Phi \in Out(G, \mathcal{O})$ is called appropriate, if for any path $w$ of $T$, there exists some positive integer $K$ s.t. for all $k \geq K$, $f^{k} _{\#} (w)$ has a splitting where the bricks are either edger or they are $N$-paths equivalent to $p$.
\end{defin}

The proof of the next lemma is the same as in the free case,  but the main difference is that we don't have finitely many paths of a given length but finitely many inequivalent paths of a given length. Therefore in the conclusion we have N-paths (and no Nielsen paths). 
\begin{lemma}\label{appropriate}
Let $\Phi \in Out(G, \mathcal{O})$ be an IWIP automorphism. Then there exists some positive power of $\Phi$, which can be represented by an appropriate train track representative.
\end{lemma}

\section{The Attractive lamination of an IWIP Automorphism}
In this section we recall the notion of an algebraic lamination. In particular, we describe the construction and the properties of the attractive lamination of an IWIP automorphism which have been proved by the author in \cite{syrigos2014irreducible}. Note that this construction is a direct generalisation of the corresponding well known notion due to Bestvina, Feighn and Handel in the free case, see \cite{bestvina1997laminations}. 
\subsection{Laminations}
We denote by $\partial ^{2} (G, \mathcal{O})$ the pairs of the boundary which don't belong in the diagonal, i.e. the set $\{ (X, Y) | X,Y \in \partial  (G, \mathcal{O}), X \neq Y \} $. Note that the topology of $\partial  (G, \mathcal{O})$ induces a topology on $\partial ^{2} (G, \mathcal{O})$. Moreover, we have a natural action of $G$ on $\partial  (G, \mathcal{O})$ which induces a diagonal action on $\partial ^{2} (G, \mathcal{O})$.
\begin{defin}
An \textbf{algebraic lamination}  $L$ of $G$ is a subset of $\partial ^{2} (G, \mathcal{O})$, which is closed, $G$-invariant and flip invariant (i.e. if $(X, Y) \in \partial ^{2} (G, \mathcal{O})$, then $(Y,X) \in \partial ^{2} (G, \mathcal{O})$).
\end{defin}
The identification of $\partial (G, \mathcal{O})$ with $\partial T$ where $T \in \mathcal{O}$, implies that the lamination $L$ induces a set of lines $L(T)$ in $T$, which we call the \textbf{symbolic lamination} in $T$-coordinates associated to  $L$. A line of the lamination is called \textbf{leaf}. Now we can define the \textbf{laminary language} $\mathcal{L} (L(T))$ in $T$-coordinates as the $G$-set of all (orbits of) finite edge paths which occur in some leaf of $L(T)$.
\subsection{Action of $Out(G, \mathcal{O})$ on the set of Laminations}
Now we can define an action of $Out(G, \mathcal{O})$ on the set of algebraic laminations of $G$, as follows: for $\phi \in Aut(G , \mathcal{O}), (X, Y) \in \partial ^{2} (G, \mathcal{O}) $ we define the map $\partial ^{2} \phi$ by $\partial ^{2} \phi (X, Y) = (\partial  \phi (X) , \partial  \phi(Y) )$ which is a well defined homeomorphism of $\partial  ^{2} (G, \mathcal{O})$ (since $\partial \phi$ is a homeomorphism). This implies that this map sends an algebraic lamination to an algebraic lamination. Note that from the $G$-invariance of the lamination follows that the image of this map depends only of the outer automorphism $\Phi$, where $\phi \in \Phi$. As a consequence we have a well defined action of the group of outer automorphisms $Out(G, \mathcal{O})$ on the set of algebraic laminations .
\begin{defin}
Let $\Phi \in Out(G, \mathcal{O})$ and $L$ be an algebraic lamination. Moreover, assume $f : T \rightarrow T$ for some $T \in \mathcal{O}$ is a topological representative of $\Phi$ and let denote by $C$ the bounded cancellation constant corresponding to $f$.
\begin{itemize}
\item We say that $\Phi$ \textbf{stabilises the algebraic lamination} $L$, if $\Phi(L) = L$.
\item We say that $f$ \textbf{stabilises the laminary language} $\mathcal{L} (L(T))$, if for all $w \in \mathcal{L} (L(T))$, $f_{\#, C} (w) \in \mathcal{L} (L(T))$.
\end{itemize}

\end{defin}
Actually, these definitions are closely related. More specifically, by \cite{syrigos2014irreducible} we have:

\begin{prop} \label{Stabilises}
Let $\Phi \in Out(G, \mathcal{O})$ be an IWIP outer automorphism and $L^{+} _{\Phi}$ be its attractive lamination. Then $\Psi$ stabilises $L^{+} _{\Phi}$ iff there is some representative $h : T \rightarrow T$ of $\Psi$, where $T \in \mathcal{O}$, which stabilises the laminary language $\mathcal{L} (L^{+} _{\Phi})$ of $\Phi$.
\end{prop}
\subsection{The Attractive Lamination of an IWIP Outer Automorphism}
Here we describe the construction of the attractive lamination relative to an IWIP outer automorphism $\Phi \in Out(G, \mathcal{O})$ and we list some interesting properties.\\
Firstly, we recall the notion of a quasi-periodic line $\ell$ of $T \in \mathcal{O}$.
\begin{defin}
 A line $\ell$ of $T$ is called \textbf{quasi-periodic} (or q.p.) if for every $L>0$ there exists some $L'$ (sufficiently large) s.t. for every subpath of length $L$ of $\ell$ occurs as subpath of some orbit of every subpath of $\ell$ of length $L'$. 
\end{defin}
Note that the notion of quasi-periodicity for $\ell$ implies that (the orbits of) every path of $\ell$ occurs infinite many times in both ends of $\ell$ and moreover the distance between any two occurrences is bounded.\\
Now for an IWIP automorphism $\Phi$, let's choose some train track representative $f: T \rightarrow T$ of $\Phi$. We can define the laminary language of our attractive lamination as the $G$- set $\mathcal{L} ^{+} _{f}$ of the orbits of finite edge paths in $T$ s.t.  an edge path $w \in \mathcal{L} ^{+} _{f}$ iff there exists an edge $e$ of $T$ and an integer $k \geq 1$ s.t. $w$ is a subpath of $f^{k}(e)$.\\
It can be proved that we have the following properties:
\begin{prop} \cite{syrigos2014irreducible} \label{properties of the lamination}
\begin{enumerate}
\item For any edge $e$ of $T$ and for all $w \in \mathcal{L} ^{+} _{f}$, there is some $k$ such that $w$ is a subpath of some orbit of $f^{k}(e)$.
\item There exists an algebraic lamination $L_{\Phi} ^{+}$ whose laminary language in $T$-coordinates is  $\mathcal{L} ^{+} _{f}$.
\item This algebraic lamination does not depend on the choice of the train track map $f$ representing $\Phi$ and of the tree $T$.
\item Every leaf of the $L^{+} _{\Phi}$ is quasiperiodic.
\end{enumerate}
\end{prop}

Moreover, we have a very interesting result about the stabiliser of the lamination which has been proved by the author in \cite{syrigos2014irreducible}:

\begin{theorem} \label{Virtually cyclic}
Let's denote by $Stab(L_{\Phi} ^{+})$ the stabiliser of the lamination. Then there is a normal periodic subgroup $A$ of $Stab(L_{\Phi} ^{+}) \cap Out(G, \{ H_i \} ^{t} ) $, such that the group $Stab(\Lambda) / A$ has a normal subgroup $B$ isomorphic to a subgroup of $\bigoplus \limits_{i=1} ^{p} Out(H_i)$ and $(Stab(\Lambda) / A) / B$ is isomorphic to $\mathbb{Z}$.
\end{theorem}

\section{Attractive Fixed points of an IWIP automorphism}
In this section, we will prove the main theorem of this paper. Our result and the method is a direct generalisation of the main result of \cite{hilion2007maximal}.
\subsection{Structure of an Attractive Fixed point of an IWIP Automorphism}
\begin{prop}\label{Pr 4.1}
Let $\Phi \in Out(G, \mathcal{O})$ be an IWIP automorphism which can be represented by an appropriate train-track map $f:T \rightarrow T$. Let $\phi \in \Phi$, and suppose that $f$ is mated with $\phi$. Let's denote by $X \in \partial (G, \mathcal{O})$ an attractive fixed point of $\phi$. Then there is some vertex $v \in T$ such that:
\begin{enumerate}
\item $[v, f^{2}  (v)] = [v, f(v)] \cdot [f(v), f^{2}(v) ]$
\item if we denote by $R_v = [v, f(v)] \cdot [f(v), f^{2} (v) ] \cdot \ldots \cdot [f^{k} (v), f^{k+1} (v)] \cdot \ldots $, we have that $R_v$ represents the point $X$
\item the segment $[v, f(v)]$ has a splitting whose bricks are either an edge or belongs to the unique equivalence class of the N- path $p$ of $T$ (if it exists) and, in addition, the first brick of this splitting is an edge.
\end{enumerate}
\end{prop}

\begin{proof}
Firstly, we will prove that we can find a point $v_0 \in T$ which satisfies the properties $(i)$ and $(ii)$.\\
By the definition of an attractive fixed point of $\partial \phi$ and since $f : T \rightarrow T$ is a train track map, there exists a vertex $v_0$ of $T$ s.t. the limit of iterates $f^{k} (v_0)$ converges to $X$.\\
We denote by $R_{v_0}$ the line that is constructed as in item $(ii)$, corresponding to $v_0$, and by our assumption we get that $R_{v_0}$ represents $X$. For every $k$ we can define inductively the points $v_k$, obtaining $v_{k+1} \in R_{v_0}$ as the projection of the reduced image of $f(v_k)$ in $R_{v_0}$. Then it is clear to see that by construction that:
\begin{equation*}
[ v_{k+1}, v_{k+2} ] \subseteq [f(v_k), f(v_{k+1} ) ] = f_{ \# } ([v_k, v+{k+1}] ) \subseteq f([v_k, v_{k+1}])
\end{equation*}
For every $k = 0,1,2,...$, we define the set $V_k =  \{ x \in T | f^{i} (v) \in [ v_{i}, v_{i+1} ], $ for every  $ 0 \leq i \leq k \}$, and since for $y \in V_k$ we get that $f^k (y) \in [v_k, v_{k+1}]$, we have that $f ^{k} (V_k) \subseteq [v_k, v_{k+1}]$. We will prove that this is actually an equality, i.e. $f ^{k} (V_k) = [v_k, v_{k+1}]$ .\\
We will prove it by induction on $n = k$:\\
For $k=0$, it is obvious since $V_0 = [v_0, v_1] = f^{0} ([v_0, v_1])  = f^{0} (V_0)$. Suppose now that our induction hypothesis is true for $n= k$, i.e. $f ^{k} (V_k) = [v_k, v_{k+1}]$ and we will prove it for $n = k+1$.\\
Since $ [ v_{k+1}, v_{k+2} ] \subseteq f( [ v_k, v_{k+1}] )$, by the induction hypothesis we get that
$[ v_{k+1}, v_{k+2} ] \subseteq f(f^{k} (V_k) ) = f^{k+1} (V_k)  $.\\
Now by definition of $V_k$, we have that:
\begin{equation*}
x \in V_{k+1} \Longleftrightarrow x \in V_k  \mbox{ and } f^{k+1} (x) \in [v_{k+1}, v_{k+2}]
\end{equation*}
But for some $y \in [v_{k+1}, v_{k+2}]$, as have seen, there is some $x \in V_k$ s.t. $ f^{k+1} (x) = y  \in [v_{k+1}, v_{k+2}] $. Using the equivalence above, we get that $x \in V_{k+1}$ and so $y \in f^{k+1} (V_{k+1})$ and our claim has been proved.\\
Then the equality $f ^{k} (V_k) = [v_k, v_{k+1}]$, implies that every $V_k$ is non-empty for every $k$. Since the $V_k$'s form a decreasing sequence of non-empty closed subsets of  $[v_0, v_1]$ (thus compact), which implies that the intersection of all $V_k$'s is non-empty. In particular, there is some point $v \in V_k$ for every $k$. By the construction of $V_k$, we get that $f^{k} (v) \in [v_k, v_{k+1}]$ for every $k$ and so $v$ satisfies the properties (i) and (ii).\\
Now we would like to prove that we can choose $v$ to be a vertex which additionally satisfies (iii). Let's suppose that $v$ is not a vertex. Then after passing to some power, we can choose an appropriate train track representative, and then the path $u = [v_0, f^2 (v_0)]$ has a splitting for which the corresponding bricks are either edges or lifts of the unique (up to equivalence) $N$-path of $f$ (if it exists). Then we consider the initial vertex $v'$ of the brick of $u$ that contains $f(v_0)$. By choice of $v'$, we have that $f^{k} _{\#} (v') \in f^{k} _{\#} (u) = [f^{k} (v_0), f^{k+2} (v_0)]$ for every $k$. Moreover, as $v' \in [v, f(v)]$, we get that $f^{k} (v') \in [f^{k} (v), f^{k+1} (v)]$. Therefore $f^{k} (v') \in R_{v_0}$ for all $k$. Therefore $v'$ can be chosen to be a vertex.\\
Finally, since $X$ is an attractive fixed point of $\partial \phi$, we have that the distance between $f^{k}(v')$ and $f^{k+1}(v')$ is going to infinity, as $k$ is going to infinity. In particular, there is a brick $b$ of $[v', f(v')]$ s.t. the length of the reduced image of $f^{k}(b)$ is going to infinity, which implies that $b$ must be an edge (since the lengths of $f^{k} _{\#} (p)$ are bounded). Changing $v'$ by the initial vertex of $b$, we find a point that satisfies (i),(i), and (iii).
\end{proof}

Now assuming the splitting of $[v, f(v)] = b_0 \cdot b_1 \cdot \ldots \cdot b_q$, as in the previous proposition, we can group together the successive bricks which are $N$-paths and they are equivalent with $p$ ( we call this new splitting, the adapted splitting). Therefore we can assume that $b_0$ is a single edge and every other $b_i$ is either a single edge (and then we say that $b_i$ is a \textbf{regular} brick) or it is equivalent to a power of $p$ (and then we say that $b_i$ is \textbf{singular}).\\
Note that, by construction, between 2 singular bricks there is at least one regular brick. Moreover, the adapted splitting of $[v,f(v)]$ induces a splitting of $[f^{k}(v), f^{k+1} (v)] = b_{0,k} \cdot b_{1,k} \cdot \ldots \cdot b_{q,k}$ for every $k$, where $b_{i,k} = f^{k} _{\#} (b_i)$. Similarly, we extend the notions of regularity and singularity, using the corresponding notions as in the adapted splitting of $[v,f(v)]$.\\
Finally, note that by construction and since every $[f ^{k}(v), f^{k+1}(v)]$ starts with a regular brick, we have that the adapted splitting of $R_v$ still satisfies the property: between 2 singular bricks there is at least one regular brick. Note also that the lengths of the singular bricks of $R_v$ are bounded uniformly (for instance this follows by the quasiperiocity of the any leaf of the lamination).

\subsection{The Stabiliser of an Attractive Fixed point of an IWIP Automorphism}

\begin{theorem} \label{Th 4.2.}
If $X \in \partial (G, \mathcal{O})$ is an attractive fixed point of an IWIP automorphism $\phi \in Aut(G, \mathcal{O})$. Assume that $\psi \in Aut(G, \mathcal{O})$ fixes $X$, then if we denote by $\Phi, \Psi$ the outer automorphisms corresponding to $\phi, \psi$ respectively, it is true that $\Psi$ stabilises the attractive lamination $L^{+} _{\Phi}$.
\end{theorem}

\begin{proof}
Firstly, we note that since $X$ is an attractive fixed point of $\phi$, then $X$ is an attractive fixed point of $\phi ^{k}$ for every $k \geq 0$. Moreover, by the construction of the (attractive) lamination (see \cite{syrigos2014irreducible}) we get again that $L_{\Phi} ^{+} = L_{\Phi ^{k}} ^{+}$ for every positive integer $k$.\\
Therefore, after possibly changing $\Phi$ with $\Phi ^{k}$, we can assume that $\Phi$ is represented (by applying \ref{appropriate}) by an appropriate train track representative 
 $f : T \rightarrow T$. Here we fix our notation, more specifically let $h : T \rightarrow T$ be the $\mathcal{O}$-map which represents $\Psi$ and let denote by $C$ the bounded cancellation constant (\ref{BCL}) corresponding to $h$. Finally, we denote by $\ell _0$ the maximal length of a singular brick in $R_v$, using the notation of the proposition \ref{Pr 4.1}.\\
Now let $u$ be an edge path of the laminary language of the symbolic attractive lamination. We need to prove that there is an occurrence (orbit) of the reduced image (after deleting some extremal paths of length $C$), $h_{\#, C} (u)$ in $R_v$ which is completely contained in a regular brick of the adapted splitting. This implies that $h_{\#, C} (u)$ is contained in $\mathcal{L} (L(T))$ and then by applying \ref{Stabilises}, the theorem follows.\\
Now since every leaf of the lamination is quasiperiodic (see \ref{properties of the lamination} for the properties of the lamination), we can find an edge path $U$ in the laminary language corresponding to $T$, which has the type $U = u u_0 u$ (where by $u$ we mean that they are of the same orbit) and $u_0$ can be chosen arbitrarily long.\\
It is convenient for us to assume that the length of $h _{\#} (u_0)$ is longer than the number $\ell + 2C$. This can be done since every $\mathcal{O}$ - map, and in particular $h$, is a quasi-isometry. (Indeed if $h$ is a $(\mu, \nu)$ quasi- isometry, it is enough to consider the starting path $u_0$ to be longer than $\mu (\ell +2C +\nu )$).\\
Using again the quasiperiodicity and the fact that the regular bricks have unbounded lengths,  we have that there is some $K$, s.t. eventually for every $k \geq K$, we can find an occurrence of $U$ in every regular brick $b_{i,k}$. In particular, we can find infinitely many occurrences of $U$ in $R_v$ and so, by the definition of the action of an automorphism on the set of laminations, infinitely many occurrences of $h _{\#,C} (U)$ in $h _{\#} (R_v)$.\\
In order to prove it, firstly we note that since $\psi(X) = X$ and so $h _{\#} (R_v) \cap R_v$ is a subray of $R_v$, there are infinitely many occurrences of $h _{\#,C} (U)$ in $R_v$. Let's denote by $w_j$ a sequence of these distinct occurrences.\\
As we have seen above, the regular bricks of $R_v$ become arbitrarily long after some steps and therefore for every path fixed path $m$ there is a finite number of occurrences of $m$ that fully contain a regular brick.\\
If there is some $w_i$ that is fully contained in a regular brick, then there is some occurrence of  $h _{\#, C} (u) \subseteq h _{\#} (U)$ which is fully contained in this brick and our claim has been proved. Otherwise, using the remark above, after passing to a subsequence, we can see that every $w_j$ meet at most two regular and a singular bricks of $R_v$.
In particular, we can suppose that are three possibly cases and we will prove that our claim is always true:
\begin{enumerate}
\item  All $w_j$'s meet two regular bricks and one singular brick which is joining the regular ones.\\
In this case, $h _{\#,C} (U) = u_1 \cdot b \cdot u_2$, where each $u_i$ is contained in some regular brick. By the choice of $\ell _0$, we can suppose that at least one of $u_i$'s satisfies the inequality: \begin{equation*}
 | u _i| \geq \frac{|h _{\#,C} (U)| -\ell _0}{2}.
\end{equation*}
Moreover, since by the definition of the map $h _{\#,C}$ and using the bounded cancellation lemma, since $U = u u_0 u$ we have that:
\begin{equation*}
|h _{\#,C} (U)  | \geq 2\cdot | h _{\#,C} (u) | + | h _{\#,C} (u_0) |.
\end{equation*}
Combining these inequalities with the choice of $u_0$, we have that $|u_i| \geq | h _{\#,C} (u) |$, which means that $ h _{\#,C} (u) $ is a subpath of $u_i$ and so it is fully contained in a regular brick.
\item  All $w_j$'s meet two consecutive regular bricks.\\
In thiss case, as above, $ h _{\#,C} (U) = u_1 \cdot u_2$, where both $u_i$'s are contained in some regular bricks. As previously, there is some $u_i$ s.t.:
\begin{equation*}
|u_i | \geq \frac{| h _{\#,C} (U) |}{2} \geq \frac{2\cdot | h _{\#,C} (u) | + | h _{\#,C} (u) | }{2} \geq | h _{\#,C} (u)|
\end{equation*}
Therefore we have the same conclusion: that $ h _{\#,C} (u) $ is contained in some $u_i$ and therefore in some regular brick.
\item All $w_j$'s meet two consecutive bricks: a regular and a singular one.\\
In this case, without loss, we assume that $ h _{\#,C} (U) = u_1 \cdot b $, where $u_1$ is contained in a regular brick and $b'$ in a singular brick. Then by the choice of $\ell_0$, we get that:
\begin{equation*}
| u_1 | \geq | h _{\#,C} (U) | - \ell _0 \geq | h _{\#,C} (u)|
\end{equation*}
As in the previous cases, we get that $h _{\#,C} (u)$ is a subpath of a regular brick.
\end{enumerate}
\end{proof}

\begin{prop}	\label{INFINITE CYCLIC}
Let $\phi \in Aut(G, \mathcal{O})$ be an IWIP automorphism. If $X$ is an attractive fixed point of $\phi$, and $\psi \in Aut(G, \{H _i\} ^{t} )$ with $\psi(X) = X$ and $\psi$ has finite order, then $\psi$ is the identity.
	
\end{prop}

\begin{proof}
Firstly, note that since $\psi$ has finite order, as in the free case we have that any fixed point of $\partial \psi$ is singular, i.e. $X \in \partial Fix(\psi)$.\\
From the Kurosh subgroup theorem we have that $Fix(\psi) = A_1 * ... * A_q * F$, where each $H_i$ is contained in some conjugate of some $G_j$ and $F$ is a free group. Since $\psi \in Aut(G, \{H _i\} ^{t} )$, it's easy to see that every $A_i$ is exactly a conjugate of the corresponding $H_j$.\\
Combining the facts that the Kurosh rank of a the fixed subgroup is less than the Kurosh rank of $G$ and the Schreier formula for Kurosh ranks (see \cite{collins1994efficient} and \cite{sykiotis2002fixed})  if $\psi$ is not the identity, then $Fix(\psi)$ is not of finite index and therefore applying the Proposition 6.2 of \cite{syrigos2014irreducible}, does not carry the lamination. In other words, if we consider some (optimal) topological representative $h : T \rightarrow T$ of $\psi$ and we denote by $T'$ the subtree of $T$ corresponding to $Fix(\psi)$, we have that $X$ is fixed by $\psi$ iff some line representing  $X$ is contained in $T'$. As we have seen, it is not possible for $T'$ to contain some leaf $\ell$ of $L ^{+ } _{\Phi} (T)$.\\

So there is a finite subpath of some $\ell \in L ^{+ } _{\Phi} (T)$ which cannot be lifted in $T'$.
Using the notation of the previous propositions, we denote by $R_v$ to be the line that represents $X$. By the definition of the lamination, $w$ appears in all sufficiently long regular bricks of $R_v$ and in particular infinitely many times in $R_v$. Also, it appears infinitely many times in any ray  $R$ in $T$ representing $X$. As a consequence, no ray representing $X$ can be lifted to $T'$. Therefore $X$ cannot be fixed by $\partial \psi$, with only exception the case where $\psi$ is the identity.

\end{proof}

Now Theorem \ref{MAIN} is just a corollary of the previous statements. In particular,
\begin{cor}
	Let $\Phi \in Out(G, \mathcal{O})$ be an IWIP outer automorphism. If $X \in \partial (G, \mathcal{O})$ is an attractive fixed point of an IWIP automorphism $\phi \in \Phi$, then $Stab(X)$ injects into $Stab(L^{+}_{\Phi})$ via the quotient map $Aut(G, \mathcal{O}) \rightarrow Out(G, \mathcal{O})$. Moreover, there is a normal subgroup $B$ of $Stab(X)$ isomorphic to  a subgroup of $\bigoplus \limits_{i=1} ^{p} Out(H_i) $ and $Stab(X) /  B$ is isomorphic to $\mathbb{Z}$.
\end{cor}

\begin{proof}
Since $\phi$ is an IWIP, and in particular it is not an inner automorphism, we have that $X$ is not a rational point, and therefore we can apply the proposition \ref{injective} and the theorem \ref{Th 4.2.}, we get that $Stab(X)$ can be seen as a subgroup of $Stab(\Lambda^{+} _{\Phi})$. Using the basic result of \cite{syrigos2014irreducible}  (Theorem \ref{Virtually cyclic}), we get that there is a normal periodic subgroup $A'$ of $Stab(X)$, such that the group $Stab(X) / A'$ has a normal subgroup $B'$ isomorphic to a subgroup of $\bigoplus \limits_{i=1} ^{p} Out(H_i)$ and $(Stab(X) / A) / B$ is isomorphic to $\mathbb{Z}$.
But then by applying \ref{INFINITE CYCLIC}, we get that $Stab(X) \cap Out(G, \{H_i\} ^{t})$ is torsion free, therefore $A'$ (which is a subgroup of $Stab(X) \cap Out(G, \{H_i\} ^{t})$ so torsion free and periodic) is the trivial group and we can conclude that there is a normal subgroup $B$ of $Stab(X)$ isomorphic to a subgroup of $\bigoplus \limits_{i=1} ^{p} Out(H_i)$ such that $Stab(X) /B$ is isomorphic to $ \mathbb{Z}$.
\end{proof}

An obvious corollary is the following:
\begin{cor}
	If $X \in \partial  (G, \{ H_1, ..., H_r \})$ is an attractive fixed point of an IWIP automorphism $\phi$ and suppose that every $Out(H_i)$ is finite, then $Stab(X)$ is virtually infinite cyclic.
\end{cor}

Example of an automorphism $\phi$ and an attractive fixed point of $X$ of $\phi$ , such that $Stab(X)$ is not cyclic.
 \begin{exa} \label{example}
	Let's suppose that our free product decomposition is of the form $G = G_1 \ast <b_1> \ast <b_2>$ , where $b_i$ are of infinite order. Here $G_1$ is an elliptic subgroup, we denote by $F_2 = <b_1> \ast <b_2>$ the "free part" and  by $\mathcal{O}$  the corresponding outer space $\mathcal{O}(G, G_1, F_2)$. Then we define the automorphism $\phi$, which satisfies $\phi(a) = a$ for every $a \in G_1$, $\phi(b_1)= b_2g_1, \phi(b_2) = b_1 b_2$ where $g_1 \in G_1$,  and it is easy to see that $\phi \in Aut(G, \mathcal{O})$ is an IWIP automorphism relative to $\mathcal{O}$. Actually, the automorphism induces a train track representative. Since $Aut(G_1)$ can be seen as a subgroup of $Aut(G, \mathcal{O})$, it follows that there is an attractive fixed point $X$ of $\phi$ which contains just the letters $b_1, b_2, g_1$ of $\phi$ and we have that for every $\psi \in Aut(G_1)$ fixes $g_1$, i.e. $\psi (g_1) = g_1$, we have that $\psi(X) = X$. Therefore $Stab(X)$ contains the subgroup $A$ of $Aut(G_1)$ of automorphisms of $G_1$ that fix $g_1$. Therefore since we can choose $G_1$ with arbitarily big $Aut(G_1)$ and in particular $A$ to not be infinite cyclic, $Stab(X)$ isn't always infinite cyclic. For example, if $G_1$ is isomorphic to $F_3$ and $g_1$ an element of its free basis, we have that $Stab(X)$ contains a subgroup which is isomorphic to $Aut(F_2)$.
\end{exa}

\bibliographystyle{plain}
\bibliography{bibfile}

\end{document}